\documentclass[11pt,reqno]{amsart}
\usepackage{amsmath}
\usepackage{amssymb}
\usepackage{amsthm}
\usepackage{enumerate}
\usepackage[usenames]{color}
\usepackage{eepic,epic}
\usepackage{url} 
\usepackage{verbatim}
\usepackage{epstopdf}
\usepackage{subfig}
\usepackage{multirow}
\usepackage{array}

\usepackage[utf8]{inputenc}

\usepackage{makecell}
\usepackage[normalem]{ulem}
\usepackage[dvipsnames]{xcolor}
\usepackage{algorithm,algpseudocode}
\usepackage{graphicx}
\usepackage{tikz, xcolor}

\newtheorem{thm}{Theorem}[section]

\newtheorem{ass}{Assumption}

\theoremstyle{definition}
\newtheorem{defn}[thm]{Definition}

\theoremstyle{remark}

\numberwithin{equation}{section}

\begin{document}
\title[]
{Adaptive consensus with exponential decay} 
\author{Woocheol Choi, Piljae Jang \\
(Sungkyunkwan University, Mathematics)}
\address{}
\email{choiwc@skku.edu}
\email{ }

\maketitle

\begin{abstract}
This paper addresses the adaptive consensus problem in uncertain multi-agent systems, particularly under challenges posed by quantized communication. We consider agents with general linear dynamics subject to nonlinear uncertainties and propose an adaptive consensus control framework that integrates concurrent learning. Unlike traditional methods relying solely on instantaneous data, concurrent learning leverages stored historical data to enhance parameter estimation without requiring persistent excitation. We establish that the proposed controller ensures exponential convergence of both consensus and parameter estimation. Furthermore, we extend the analysis to scenarios where inter-agent communication is quantized using a uniform quantizer. We prove that the system still achieves consensus up to an error proportional to the quantization level, with exponential convergence rate.

\end{abstract}

\section{Introduction}\label{section1}
Distributed control has recently gained significant attention across various fields. In areas like autonomous vehicles, robotic networks, biochemical networks, and social networks, it enables coordinated behavior among multiple agents \cite{CS}.

An important problem in the study of consensus control is to address the uncertainty in the dynamic of each agent. In this work, we are interested in the consensus problem for the  uncertain dynamic system involving a matched nonlinear uncertainty which was studied in the recent works \cite{YBC, MRS, YMM}. The work \cite{YBC} proposed a consensus controller updating the unknown parameters. They showed the parameter estimation and the consensus are achieved asymptotically. The work \cite{MRS} proposed a model reference adaptive controller for the consensus for the undirected graph and also modified the controller for the directed graph using a slack variable. The argument in the work \cite{MRS} for the consensus for the undirected graph case relies on the Lyapunov analysis and the Barbalet's lemma, and so the convergence rate for the consensus and the parameter estimation is missing. We also refer the work \cite{YMM} which appled a reinforcement learning approach to attain the bipartite conensus of the agents.

The aim of this work is to design an adaptive consensus control for the system considered in \cite{MRS} which is guaranteed to achieve the consensus exponentially fast. For this we make use of the concurrent learning developed by Chowdhary and Johnson \cite{CJ}. Concurrent learning leverages historical data to improve parameter estimation accuracy and control performance \cite{CJ}. Unlike conventional adaptive control, which relies solely on current data, concurrent learning stores past data points and continuously integrates them into the learning process. This approach addresses the limitations of traditional adaptive methods by significantly improving the accuracy of parameter estimates, especially when  the persistent excitation condition is difficult to maintain.

Several studies have addressed consensus problems for uncertain multi-agent systems.  We refer to \cite{YMM} which handles the dynamic system with matched uncertainty structue such that the  matrices are also unknown.  We also refer to the problem which considered multi-agent systems with linear dynamics involving unknown matrices \cite{YSHR} \cite{YBD}.   Chen et al. \cite{CLRW} introduced an adaptive consensus method using a Nussbaum-type function for systems with unknown control directions. Hou et al. \cite{HCT} proposed a decentralized robust adaptive control approach employing neural networks to manage uncertainties. Chen et al. \cite{CWLW} developed an adaptive consensus control method for nonlinear multi-agent systems with time delays, leveraging neural networks for robustness. Ma et al. \cite{MWWL} presented a neural-network-based distributed adaptive robust control scheme addressing time delays and external noises. Psillakis \cite{P} introduced a PI consensus error transformation for adaptive cooperative control of nonlinear multi-agent systems, ensuring consensus despite uncertainties. Cai and Huang \cite{CH} tackled leader-following consensus for Euler-Lagrange systems with uncertainties using an adaptive distributed observer. We also refer  \cite{ZWLA} \cite{ZZ} which make use of the neural-network for the consensus problem.

We also obtain the consensus result for the uncertain multi-agent system when the agents use the quantization for the communication. 
Quantization is important for the multi-agent system as it reduces the communication burden significantly. 
To enhance the robustness and reliability of adaptive consensus algorithms under quantized data transmission, several studies have specifically addressed the quantization problem in this context. Ishii and Tempo \cite{IT} analyzed quantized consensus algorithms, focusing on the trade-offs between quantization levels and convergence speed, and demonstrated that appropriate quantization scheme design can lead to efficient consensus despite limited communication precision. Lavaei and Murray \cite{LM} proposed a quantized consensus algorithm using the gossip approach, mitigating the effects of quantization errors and improving overall system performance. Zhang et al. \cite{ZZHW} explored leader-following consensus for both linear and Lipschitz nonlinear multi-agent systems with quantized communication, introducing an event-triggered control mechanism to enhance system efficiency and reliability under bandwidth constraints.

The rest of this paper is organized as follows. In Section 2, we introduce the perliminaries and state the main results of this paper. Section 3 is devoted to prove the exponential convergence result for the proposed consensus control. In Section 4 we study the adaptive consensus control for the multi-agent system involving the quantization for the communication. Section 5 is devoted to present numerical experiments supporting the theoretical results of this paper.



\section{Preliminaries}\label{section2}

In this section, we state the distributed consensus problem involving uncertainty in the system of the agents. We begin with recalling the graph theory to state the problem.

\subsection{Graph theory}

Consider an undirected network represented by the graph $\mathcal{G} = (\mathcal{V}, \mathcal{E})$, where $\mathcal{V}$ denotes the set of vertices $\mathcal{V} = \{1, \ldots, n\}$, and $\mathcal{E}$ denotes the set of edges, $\mathcal{E} \subseteq \mathcal{V} \times \mathcal{V}$. An edge $(i, j) \in \mathcal{E}$ indicates a pair of agents capable of communication, and in the context of an undirected network, $(i, j) \in \mathcal{E}$ if and only if $(j,i)\in \mathcal{E}$. The neighborhood set of an agent $i$ is denoted by $\mathcal{N}_i = \{j \in \mathcal{V} \mid (i, j) \in \mathcal{E}\}$, representing the set of neighbors of agent $i$. We assume that the graph $\mathcal{G}$ to be considered connected, there must exist a path between any pair of vertices within the graph.

The adjacency matrix $\mathcal{A}(\mathcal{G}) = [a_{ij}]_{i,j=1}^n$ is used to indicate communication pairs, where $a_{ij} = 1$ if $(i, j) \in \mathcal{E}$, and $a_{ij} = 0$ otherwise. The degree matrix $\Delta(\mathcal{G}) = \mathrm{diag}(d_1, \ldots, d_n)$ is a diagonal matrix where each diagonal element represents the sum of the elements in the $i$-th row (or column) of $\mathcal{A}(\mathcal{G})$, thus $d_i = \sum_{j} a_{ij} = |\mathcal{N}_i|$.

The Laplacian matrix is defined as $\mathcal{L}(\mathcal{G}) = \Delta(\mathcal{G}) - \mathcal{A}(\mathcal{G})$.  This matrix always has an eigenvalue of 0, corresponding to the eigenvector from the family spanned by $\mathbf{1}_N$.

\begin{defn}
For a connected graph $\mathcal{G} =(\mathcal{V}, \mathcal{E})$ with the Laplacian matrix $\mathcal{L}$, the general algebraic connectivity is defined by 
\begin{equation}
	\alpha(\mathcal{L})=\min\limits_{x^T \xi =0, x\neq=0} \frac{x^T \mathcal{L} x}{x^T x}, \quad  \xi = 1_p.
\end{equation}
Furthermore $\alpha(\mathcal{L})=\lambda_2(\mathcal{L})$, where $\lambda_2(\mathcal{L})$ is the smallest positive eigenvalue of the Laplacian matrix $\mathcal{L}$.
\end{defn}

\subsection{Problem statement}
Consider the general linear dynamics of each agent which is described by
\begin{equation}\label{eq-2-1}
\dot{x}_i (t)=Ax_i (t)+B(u_i (t)+\Phi_i(t,x_i (t))\theta_i (t)), \ i=1,....,n 
\end{equation}
where $x_i\in\mathbb{R}^p$ is the state of agent $i$, $u_i\in\mathbb{R}^q$ is the control function, $A\in\mathbb{R}^{p\times{p}}$ and $B\in\mathbb{R}^{p\times{q}}$ are constant matrices, $\Phi_i\in\mathbb{R}^{q\times{m}}$ is a bounded continuous function and $\theta_i\in\mathbb{R}^m$ is an unknown constant parameter.

We are interested in constructing a feedback control for the system \eqref{eq-2-1} that achieves the asymptotic consensus.

We assume the following condition for $(A,B)$.
\begin{ass} The pair $(A,B)$ is stabilizable. 
\end{ass}

This problem was studied in the paper \cite{MRS} where the authors proposed the following adaptive control:
\begin{equation}\label{eq-2-10}
\begin{split}
&u_i (t)=\alpha{K}\sum^n_{j=1}a_{ij}(x_i (t)-x_j (t))-\Phi_i(t,x_i)\hat{\theta}_i (t).\\
&\dot{\hat{\theta}}_i (t)= \Phi^T_iB^TP\sum^n_{j=1}a_{ij}(x_i (t)- x_j (t)).
\end{split}
\end{equation}
where $\alpha>0$ and $K=-B^{T}P \in \mathbb{R}^{q\times p}$. Here $P \in \mathbb{R}^{p\times p}$ is selected as a positive symmetric matrix solving the following Algebraic Riccati Equation:
\begin{equation}
	A^{T}P+PA-PBB^{T}P+I_p=0.
\end{equation}
 
Then \eqref{eq-2-1} can be written as
\begin{equation}\label{eq-3-1}
	\dot{x}_i (t)=Ax_i (t)+B\bigg[\alpha{K}\sum^n_{j=1}a_{ij}(x_i (t)-x_j (t))-\Phi_i(t,x_i(t))\tilde{\theta}_i (t)\bigg]
\end{equation}
where $\tilde{\theta}_i (t)=\hat{\theta}_i (t)-\theta_i$.

We can write this in a vectorial form as
\begin{equation}
	\dot{x}(t) = (I_n\otimes{A})x (t)+ \alpha(\mathcal{L}\otimes{BK})x (t)- (I_n\otimes{B})\Phi\tilde\Theta (t)\\.
\end{equation}
It was proved in \cite{MRS} that the system \eqref{eq-2-1} with control \eqref{eq-2-10} achieves the consensus as the time goes to infinity. The result guarantees the consensus but the convergence rate is missing, which is important in practical scenarios. The aim of this paper is to propose an adaptive controller for the system \eqref{eq-2-1} that may guarantee the exponential convergence rate.  For this we make use of the concurrent learning developed in the work \cite{CJ}.

\medskip

\noindent \text{\textbf{Condition 1.}} For each $i \in \{1,2,\cdots, m\}$, the matrix $Z_i=[\Phi(x_{i,1}),....,\Phi(x_{i,p})]$ has $rank(Z_i)=m$.

\medskip

Precisely, we propose the following update rule for the uncertain variable:
\begin{equation}\label{eq-2-2}
\dot{\hat{\theta_i}}(t) = \Phi^T_iB^TP\sum^n_{j=1}a_{ij}(x_i (t)- x_j (t)) - \sum^r_{k=1}\Phi_i^T(x_{i,k})\Phi_i(x_{i,k})\tilde\theta_i (t),
\end{equation}
where $\tilde{\theta}_i (t)=\hat{\theta}_i (t)-\theta_i$.

 We will show that the control \eqref{eq-2-12} along with \eqref{eq-2-2} achieves the consensus for the system \eqref{eq-2-1} exponentially fast. 

\begin{thm}\label{thm-1}Assume that the general linear dynamics of each agent\ is described by \eqref{eq-2-1} and Condition 1 holds.
Set 
\begin{equation}\label{eq-3-2}
\dot{\hat{\theta_i}}(t) = \Phi^T_iB^TP\sum^n_{j=1}a_{ij}(x_i (t)- x_j (t)) - \sum^r_{k=1}\Phi_i^T(x_{i,k})\Phi_i(x_{i,k})\tilde\theta_i (t).
\end{equation}
Then the dynamics \eqref{eq-2-1} is exponentially stable.
\end{thm}

\subsection{Quantized communcation}
As a benefit of our framework, we extend our analysis for the case that the communication is done by a quantization, which is more practical for real-world scenarios. Let us consider a uniform quantizer $q_u:R\rightarrow\sigma Z$ in this paper, which can be defined by 
\begin{equation}\label{eq-2-3}
q_u(x)=\lfloor \frac{x}{\sigma} + \frac{1}{2} \rfloor \sigma
\end{equation}
where $\sigma$ is a positive number. We let
\begin{equation}
q_u(x)-x=\Delta_u(x).
\end{equation}
Note that $\vert q_u(x)-x \vert \leq (\frac{\sigma}{2})$, which reflects accuracy of quantizer. We can write this in a vectorial form as
\begin{equation}
\Vert q_u(x)-x \Vert \leq n(\frac{\sigma}{2})
\end{equation}
for $x=[x_1, x_2,..., x_n]^T$ and $q_u(x)=[q_u(x_1), q_u(x_2),..., q_u(x_n)]^T$.\\

We consider the following control
\begin{equation}\label{eq-2-12} 
u_i (t)= \alpha{K}\sum^n_{j=1}a_{ij}(q_u(x_i (t)) - q_u(x_j (t))) - \Phi_i(t,x_i (t))\hat\theta_i (t),
\end{equation}
which is companied with the following update rule:
\begin{equation}\label{eq-2-4}
\dot{\hat{\theta}}_i  (t)= \Phi^T_iB^TP\sum^n_{j=1}a_{ij}(q_u(x_i (t)) - q_u(x_j (t))) - \sum^r_{k=1}\Phi_i^T(x_{i,k})\Phi_i(x_{i,k})\tilde\theta_i (t).
\end{equation}
We will show that the control \eqref{eq-2-12} with the update rule \eqref{eq-2-4} achieves the consensus up to an error $O(\sigma)$ exponentially fast for the system \eqref{eq-2-1}.

 .

\

\section{Main Result}\label{section3}

In this section, we propose a consensus control for the system \eqref{eq-2-1} involving the concurrent learning. For the proposed control \eqref{eq-3-2}, we establish that the controlled system is exponentially stable.
 
\begin{proof}[Proof of Theorem \ref{thm-1}]We consider the following Lyapunov function 
\begin{equation}
V(t)=x(t)^T(L_A\otimes{P})x(t)+\frac{1}{2}\sum^n_{i=1}tr(\tilde\theta^T_i\tilde\theta_i)(t).
\end{equation}
Differentiating this, we find
\begin{equation}\label{eq-3-20}
\begin{split}
\dot{V}(t)&=\dot{x(t)}^T(L_A\otimes{P})x(t)+ x(t)^T(L_A\otimes{P})\dot{x(t)} + \frac{1}{2}\sum^n_{i=1}tr(\dot{\tilde{\theta_i}}^T\tilde{\theta_i} + \tilde{\theta_i}^T \dot{\tilde{\theta_i}})\\
&=x(t)^T(L_A \otimes A^T P)x(t) + \alpha x(t)^T(L_A^2 \otimes K^TB^TP)x(t) -\tilde{\Theta}^T \Phi^T(L_A \otimes B^TP)x(t) \\
&\quad +x(t)^T(L_A \otimes PA)x(t) + \alpha x(t)^T(L_A^2 \otimes PBK)x(t) -x(t)^T(L_A\otimes PB)\Phi \tilde{\Theta} + \sum^n_{i=1}tr(\tilde{\theta_i}^T\dot{\tilde{\theta_i}})\\
&=x(t)^T(L_A \otimes( AP^T+PA))x(t) +\alpha x(t)^T (L_A^2 \otimes (K^TB^TP + PBK))x(t)\\
&\quad -2\tilde{\Theta}^T \Phi^T (L_A \otimes B^TP)x(t) +\sum^n_{i=1}tr(\tilde{\theta_i}^T\dot{\tilde{\theta_i}})\\
&= - x(t)^T(L_A\otimes{I_P})x(t) - x(t)^T[(2\alpha{L^2_A} - L_A)\otimes{PBB^TP}]x(t)  
\\
&\qquad\quad + \sum^n_{i=1}tr\left[\tilde\theta^T_i\bigg(\dot{\hat{\theta_i}} - \Phi^T_iB^TP\sum^n_{j=1}a_{ij}(x_i (t) - x_j (t))\bigg)\right],
\end{split}
\end{equation}
where we used \eqref{eq-3-1} in the second equality.
Inserting \eqref{eq-3-2} into the above formula, we get
\begin{equation}\label{eq-3-5}
\begin{split}
\dot{V}(t) &= - x(t)^T\{(L_A\otimes{I_P}) + [(2\alpha{L^2_A} - L_A)\otimes{PBB^TP}]\}x(t)
\\
&\quad -\sum^n_{i=1}tr\left[\tilde\theta^{T}_{i}\bigg(\sum^r_{j=1}\Phi_i^T(x_{i,k})\Phi_i(x_{i,k})\bigg)\tilde\theta_i\right].
\end{split}
\end{equation}
By Condition 1, there exists a positive value $q_i$ such that 
\begin{equation}
\sum^p_{j=1}\Phi_i^T(x_{i,k})\Phi_i(x_{i,k})\geq q_i \cdot I_m
\end{equation}
and so we have
\begin{equation}
\sum^n_{i=1}tr\left[\tilde\theta^T_{i}\bigg(\sum^p_{j=1}\Phi_i^T(x_{i,k})\Phi_i(x_{i,k})\bigg)\tilde\theta_i\right] \geq q\sum^n_{i=1}tr(\tilde\theta^T_i\tilde\theta_i),
\end{equation}
where $q = \min_{1\leq i \leq n} q_i$. For $\alpha \geq \frac{1}{2 \min_{\mu_i \neq 0}\mu_i}$, the matrix $2\alpha L_A^2 -L_A$ is positive semi-definite.   In addition, we have
\begin{equation}\label{eq-3-6}
\begin{split}
x(t)^T (L_A \otimes P) x(t) & = (x(t)-\bar{x(t)})^T (L_A \otimes P) (x(t)-\bar{x(t)})
\\
&\leq C \|x(t)-\bar{x(t)}\|^2
\\
&\leq (C/\alpha (L_A)) (x(t)-\bar{x(t)})^T (L_A \otimes I_p) (x(t)-\bar{x(t)})
\\
& =(1/\gamma) x(t)^T (L_A \otimes I_p) x(t),
\end{split}
\end{equation}
where we have let $\gamma = \alpha (L_A)/C$.  Therefore we have
\begin{equation}\label{eq-3-7}
x(t)^T\{(L_A\otimes{I_P}) + [(2\alpha{L^2_A} - L_A)\otimes{PBB^TP}]\}x(t) \geq \gamma x(t)^T(L_A\otimes{P})x(t).
\end{equation} 
Using above estimates in \eqref{eq-3-5}, we get
\begin{equation}
\begin{split}
\dot{V}(t)&\leq-\gamma x(t)^T\{(L_A\otimes{P})x(t) -q\sum^n_{i=1}tr(\tilde\theta^T_{i}\tilde\theta_{i}) (t)\\
&\leq-min\{\gamma,q\}V(t).\\
\end{split}
\end{equation}
It implies that $V(t) \leq e^{-t \cdot min\{\gamma,q\}} V(0)$, estabilishing the exponential consensus of the system.
\end{proof}

\section{Quantized version}

In this section, we study the system \eqref{eq-2-1} with control \eqref{eq-2-12} involving the quantization operator. Due to the discontinuity of the quantization operator, the right-hand side of the controlled dynamic equation contains a discontinuous term. Hence we need to consider the Krasovskii solution as in the work \cite{ZZHW}.
 
Consider the following differential equation
\begin{equation}\label{discon}
\dot{x}(t)=f(x(t)), \ x(t_0)=x_0
\end{equation}
with the discontinuous right hand side, where $f(\cdot):\mathbf{R}^n\rightarrow\mathbf{R}^n$ is a discontinuous function. Because of the discontinuity of \eqref{discon}, a classical solution does not work.  

We say that $x:[t_0,t_0 +a] \rightarrow \mathbf{R}^n$ for a value $a>0$ solves the equation \eqref{discon} in the Krasovskii sense if $x( \cdot)$ is absolutely continuous and for almost every time $t\in [t_0, t_0 +a]$ satisfies the following differential inclusion   
\begin{equation}
\dot{x}(t) \in \mathcal{K}[f(x(t))]\triangleq\cap_{v>0}\bar{co}f(O(x,v))
\end{equation}
where $\bar{co}$ denotes the convex closure, $f(O(x, v))$ is assessed at the points within $O(x, \nu)$, which represents the open ball centered at $x$ with radius $\nu$. It is known that a local Krasovskii solution of equation \eqref{discon} exists if $f(\cdot) : \mathbb{R}^n \rightarrow \mathbb{R}^n$ is measurable and locally bounded.

Next, we consider the case that the communication between agents of \eqref{eq-2-1} involves the quantization.

\begin{thm}\label{thm-2} (Quantized version) Consider the general linear dynamics \eqref{eq-2-1} with the control \eqref{eq-2-12} involving the quantization with level $\sigma >0$ defined as

\begin{equation}\label{eq-2-3}
q_u(x)=\lfloor{x/\sigma} + 1/2\rfloor\sigma
\end{equation}
and set
\begin{equation}\label{eq-2-4}
\dot{\hat{\theta_i}}(t) = \Phi^T_iB^TP\sum^n_{j=1}a_{ij}(q_u(x_i) - q_u(x_j)) - \sum^p_{j=1}\Phi_i^T(x_{i,k})\Phi_i(x_{i,k})\tilde\theta_i (t)\\
\end{equation}
Then the dynamics \eqref{eq-2-1} with uniform quantizer \eqref{eq-2-3} is exponentially stable up to an $O(\sigma)$-error.
\end{thm}

\begin{proof}Inserting \eqref{eq-2-12} into \eqref{eq-2-1}, we find
\begin{equation}
\begin{split} 
\dot{x_i}(t) &= Ax_i(t)+B[\alpha{K}\sum^n_{j=1}a_{ij}(q_u(x_i (t))-q_u(x_j (t))) - \Phi_i\tilde\theta_i (t)]\\ 
&= Ax_i(t)+B[\alpha{K}\sum^n_{j=1}a_{ij}(x_i (t) + \vartriangle_i (t) - x_j (t) - \vartriangle_j (t)) - \Phi_i\tilde\theta_i (t)].\\
\end{split}
\end{equation}
where $\vartriangle_i (t)= q_u(x_i (t)) - x_i (t)$. This equation admits a Krasovskii solution $(x_1 (t), \cdots, x_n (t)) \in \mathbb{R}^n$ such that
\begin{equation}
\dot{x_i}(t) \in Ax_i (t)+ B[\alpha{K}\sum^n_{j=1}a_{ij}(x_i (t) - x_j (t)) - \Phi_i\tilde\theta_i (t)] + B\alpha{K}\sum^n_{j=1}a_{ij}(\mathcal{K}[\vartriangle_i](t) - \mathcal{K}[\vartriangle_j](t)),
\end{equation}
and so  there are values $v_i (t) \in \mathcal{K}[\vartriangle_i](t)$ and $v_j (t)\in \mathcal{K}[\vartriangle_j](t)$ such that
\begin{equation}
\dot{x_i}(t) = Ax_i (t) + B\Big(\alpha{K}\sum^n_{j=1}a_{ij}(x_i (t) - x_j (t)) - \Phi_i\tilde\theta_i (t)\Big) + B\alpha{K}\sum^n_{j=1}a_{ij}(v_i (t)- v_j (t))\\.
\end{equation}
We write this in a vectorial form as follows
\begin{equation}\label{eq-4-1}
\dot{x}(t) = (I\otimes{A})x (t)+ \alpha(L_A\otimes{BK})x(t) + \alpha(L_A\otimes{BK})v(t) - (I\otimes{B})\Phi\tilde\Theta (t)\\.
\end{equation}
Now we consider the Lyapunov function $V:[0,\infty) \rightarrow [0,\infty)$ given by
\begin{equation}
V(t)=x(t)^T(L_A\otimes{P})x(t)+\frac{1}{2}\sum^n_{i=1}tr(\tilde\theta^T_i\tilde\theta_i)(t)\\
\end{equation}
By differentiating this,
\begin{equation}
\dot{V}(t)=\dot{x}(t)^T(L_A\otimes{P})x(t) + x(t)^T(L_A\otimes{P})\dot{x}(t) + \frac{1}{2}\sum^n_{i=1}tr(\dot{\tilde{\theta_i}}^T\tilde{\theta_i} + {\tilde{\theta_i}}^T\dot{\tilde{\theta_i}})(t)\\
\end{equation}
We let
\begin{equation}
Q(t) =\dot{x}(t)^T(L_A\otimes{P})x(t) + x(t)^T(L_A\otimes{P})\dot{x}(t)\\
\end{equation}
Using \eqref{eq-4-1} and following \eqref{eq-3-20}  we find
\begin{equation}
Q(t) =- x(t)^T(L_A\otimes{I_P})x(t) - x(t)^T[(2\alpha{L^2_A} - L_A)\otimes{PBB^TP}]x(t) + E(t)
\end{equation}
where
\begin{equation}
E (t) =\alpha v(t)^T(L_A\otimes{BK})^T(L_A\otimes{P})x(t) + \alpha x(t)^T(L_A\otimes{P})(L_A\otimes{BK})v(t).\\
\end{equation}
We note that $E(t)$ involves the term $v$ which comes from the quantization error.  Using $K= -B^T P$ we reformulate $E(t)$ as follows
\begin{equation}
\begin{split}
E(t)&=  \alpha v(t)^T(L_A^2\otimes{K^TB^TP})x(t) + \alpha x^T(L_A^2\otimes{PBK})v(t)\\
&= - x(t)^T(2\alpha{L_A}^2\otimes{PBB^TP})v(t).\\
\end{split}
\end{equation}
 Combining the above computations, we have
\begin{equation}
\begin{split}
\dot{V}(t)=&- x(t)^T(L_A\otimes{I_P})x(t) - x^T[(2\alpha{L^2_A} - L_A)\otimes{PBB^TP}]x(t)+E(t)\\
&+ \sum^n_{i=1}tr\left[\tilde\theta^T_i\bigg(\dot{\hat{\theta_i}} - \Phi^T_iB^TP\sum^n_{j=1}a_{ij}(q_u(x_i (t)) - q_u(x_j (t)))\bigg)\right].\\
\end{split}
\end{equation}
Inserting \eqref{eq-2-4} here, we get
\begin{equation}\label{eq-4-7}
\begin{split}
\dot{V}(t) &= - x(t)^T\{(L_A\otimes{I_P}) + [(2\alpha{L^2_A} - L_A)\otimes{PBB^TP}]\}x(t) +E(t)\\ 
&\quad- \sum^n_{i=1}tr\left[\tilde\theta{_i}{^T}\bigg(\sum^p_{j=1}\Phi_i^T(x_{i,k})\Phi_i(x_{i,k})\bigg)\tilde\theta_i\right].\\
\end{split}
\end{equation}
We proceed to estimate the right hand side of the above inequality. As in \eqref{eq-3-7} we have
\begin{equation}
\begin{split}
x(t)^T\{(L_A\otimes{I_P}) + &[(2\alpha{L^2_A} - L_A)\otimes{PBB^TP}]\}x(t)\geq \gamma x(t)^T (L_A ^2 \otimes PBB^T P)x(t)
\end{split}
\end{equation}
and by Condition 1 we have
\begin{equation}
\sum^n_{i=1}tr\left[\tilde\theta^T_i\bigg(\sum^p_{j=1}\Phi_i^T(x_{i,k})\Phi_i(x_{i,k})\bigg)\tilde\theta_i\right] \geq  q\sum^n_{i=1}tr(\tilde\theta^T_i\tilde\theta_i).
\end{equation}
We note that 
\begin{equation}
\begin{split}
E(t) &= (x(t) - \bar{x}(t))^T (2\alpha L_A^2 \otimes PBB^T P) v(t)
\\
&\leq  \alpha D \|x(t) - \bar{x}(t)\| \|v(t)\|,
\end{split}
\end{equation}
where we have let $D = \| L_A^2 \otimes PBB^T P\|$. Also we the inequalities in \eqref{eq-3-6} to find
\begin{equation}
\begin{split}
x(t)^T (L_A \otimes I_p) x(t)& \geq \frac{\gamma}{2} x(t)^T (L_A \otimes P) x(t) + \frac{\alpha (L_A)}{2}\|x(t) -\bar{x}(t)\|^2.
\end{split}
\end{equation}
We can formulate the inequality such that 
\begin{equation}\label{eq5}
\begin{split}
\dot{V}(t) \leq &-\frac{\gamma}{2} x(t)^T (L_A  \otimes P)x(t) -  \frac{\alpha (L_A)}{2}\|x(t) -\bar{x}(t)\|^2
\\
&\quad - q\sum^n_{i=1}tr(\tilde\theta^T_i\tilde\theta_i)(t) +\alpha D \|x(t) - \bar{x}(t)\| \|v(t)\|.
\end{split}
\end{equation}
Using Young's inequality here, we get
\begin{equation}
\begin{split}
\dot{V}(t) \leq &-\frac{\gamma}{2} x(t)^T (L_A  \otimes P)x(t) - q\sum^n_{i=1}tr(\tilde\theta^T_i\tilde\theta_i )(t)+ J   \|v(t)\|^2,
\\
&  \leq -min\{\gamma/2,q\}V(t) +  J \|v(t)\|^2. 
\end{split}
\end{equation} 
where $J = \alpha^2 D^2 / \alpha (L_A)$.  It implies that
\begin{equation}
\begin{split}
V(t)  & \leq  e^{-t \cdot min\{\gamma/2,q\}} V(0) + \int_0^{t} e^{-\min\{\gamma/2, q\}(t-s)} J \|v(s)\|^2 ds
\\
&< e^{-t \cdot min\{\gamma/2,q\}} V(0) + \frac{J \sigma^2}{\min \{\gamma/2, q\}}.
\end{split}
\end{equation}
This shows that the system achieves the consensus up to an $O(\sigma)$-error with an exponential rate. The proof is done.
\end{proof}

\section{Simulation}\label{section5}
In this section, we present numerical tests supporting the theoretical results of this paper. In the dynamics \eqref{eq-2-1}, we set $p=4$ and $q=2$, and choose the matrices $A\in \mathbf{R}^{p\times p}$ and $B\in \mathbf{R}^{p\times q}$ as follows:
\begin{equation*}
	A = 
	\begin{pmatrix}
		-0.8& 0.3 & 0.2 & 1.1 \\
		0.4 & -0.5 & 1.2 & 0.6 \\
		0.7 & 0.9 & 0.2 & 0.5\\
		1.3 & 1.1 & 0.4 & -0.1 
	\end{pmatrix}
	,B=
	\begin{pmatrix}
		1.2& 0.7 \\
		0.6 & 1.3\\
		1.1 & 1.4\\
		0.9& 1.2 
	\end{pmatrix}.
\end{equation*}
This pair $(A,B)$ is stabilizable. We consider the undirected multi-agent system with 5 agents which exchange information via a communication graph depicted in Figure \ref{com.graph}. 
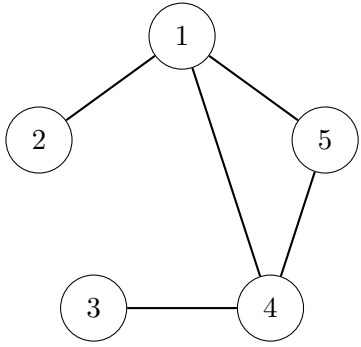
\begin{figure}[h]
	\centering
	\begin{tikzpicture}
		\tikzstyle{vertex} = [circle, draw, minimum size=25pt, inner sep=0pt]
		
		\node[vertex] (1) at (90:2) {1};
		\node[vertex] (2) at (162:2) {2};
		\node[vertex] (3) at (234:2) {3};
		\node[vertex] (4) at (306:2) {4};
		\node[vertex] (5) at (18:2) {5};
		
		\draw[thick] (1) -- (2);
		\draw[thick] (1) -- (5);
		\draw[thick] (4) -- (5);
		\draw[thick] (3) -- (4);
		\draw[thick] (1) -- (4);
	\end{tikzpicture}
	\caption{Undirected communication graph}
	\label{com.graph}
\end{figure}

We use the same Laplacian matrix and the nonlinear function $\Phi$ as in \cite{CSHD}. Specifically, the Laplacian matrix $L\in \mathbf{R}^{5\times 5}$ is given by
\begin{equation*}
	\mathcal{L}=
	\begin{pmatrix}
		2.168 & -1.037 & 0 & -0.865 & -0.266\\
		-1.037&1.037&0&0&0\\
		0&0&1.651&-1.651&0\\
		-0.865&0&-1.651&2.863&-0.347\\
		-0.266&0&0&-0.347&0.613\\
	\end{pmatrix},
\end{equation*}
and the solution of Algebraic Riccati Equation $P\in\mathbf{R}^{4\times4}$ is 
\begin{equation*}
	P=
	\begin{pmatrix}
		2.8917 & -0.3741 & -1.8010 & 1.2765\\
		-0.3741 & 0.5278 & 0.3487 & -0.0738\\
		-1.8010 & 0.3487 & 2.1217 & -1.0746\\
		1.27656 & -0.0738 & -1.0746 & 1.1882
	\end{pmatrix}.
\end{equation*}
We choose $q=4$ and $m=1$ in \eqref{eq-2-1} and the nonlinear function $\Phi_d\in \mathbf{R}^2$ for $d=\{1,2,3,4,5\}$ is defined as
\begin{equation*}
	\Phi_d(x_d(t),t)=
	\begin{pmatrix}
		\gamma_d + \beta_d e^{-d}x_{d1}\\
		\gamma_d + \beta_d e^{-d}x_{d2}
	\end{pmatrix}
\end{equation*}
where
\begin{equation*}
	\begin{split}
		\gamma&=(\gamma_1, \gamma_2, \gamma_3, \gamma_4, \gamma_5)
		=
		\begin{pmatrix}
			0.4157 & 0.4017 & 0.0302 & 0.1996 & 0.2634
		\end{pmatrix}\\
		\beta&=(\beta_1, \beta_2, \beta_3, \beta_4, \beta_5)
		=
		\begin{pmatrix}
			0.3437 & 0.5474 & 0.5233 & 0.2433 & 0.3597
		\end{pmatrix}.
	\end{split}
\end{equation*}
The references of parameters are set as
\begin{equation}
	\theta = (\theta_1, \ \theta_2, \ \theta_3, \ \theta_4, \ \theta_5 )=
	\begin{pmatrix}
		3 & 6 & 1.5 & 5.5 & 0.5	
	\end{pmatrix}.
\end{equation}
The initial states of the agents for the simulation are given by
\begin{equation}
	\begin{split}
		x_1(0) &= [-6.125, \ 4.375, \ 9.875, \ -8.125]^T \\
		x_2(0) &= [1.0, \ -8.5, \ -2.5, \ 10.0]^T \\
		x_3(0) &= [-5.0, \ 7.5, \ -3.5, \ 1.0]^T \\
		x_4(0) &= [10.125, \ -3.875, \ -2.875, \ -3.375]^T \\
		x_5(0) &= [1.125, \ 0.125, \ -0.875, \ -0.375]^T
	\end{split}
\end{equation}
The initial values of the variable $\{\hat\theta_i(t)\}^5_{i=1}$ targeting $\theta$ are
\begin{equation}
\hat\theta(0) = (\hat\theta_1(0), \hat\theta_2(0), \ \hat\theta_3(0), \ \hat\theta_4(0), \ \hat\theta_5(0))=
\begin{pmatrix}
1 & 1 & 3 & 2 & 5	
\end{pmatrix}.	
\end{equation}
We simulate \eqref{eq-2-1} involving the control given by \eqref{eq-2-10}. with $\alpha=0.8019$ and $K=-B^TP$.
The graph of consensus error $\sum_{i,j \in N} ||x_i(t) - x_j(t) ||^2$ and the graphs of the values $\hat\theta(t)$ are depicted in Figure \ref{ep}, which indicates that the consensus achieved with an exponential rate. Also it shows that $\hat\theta(t)$ converges to $\theta$ exponentially fast as proved in Theorem \ref{thm-1}.  Figure \ref{c} exhibits the graphs of the first and third variables of the state of the agents, revealing that the consensus is acheived for those coordinates.\\
\begin{figure}[h]
	\centering
	\begin{minipage}{0.45\textwidth}
		\centering
		\includegraphics[width=\textwidth]{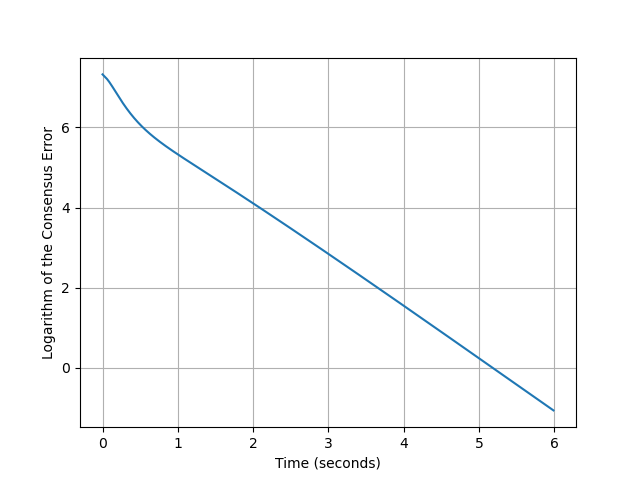}
	\end{minipage}
	\hfill
	\begin{minipage}{0.45\textwidth}
		\centering
		\includegraphics[width=\textwidth]{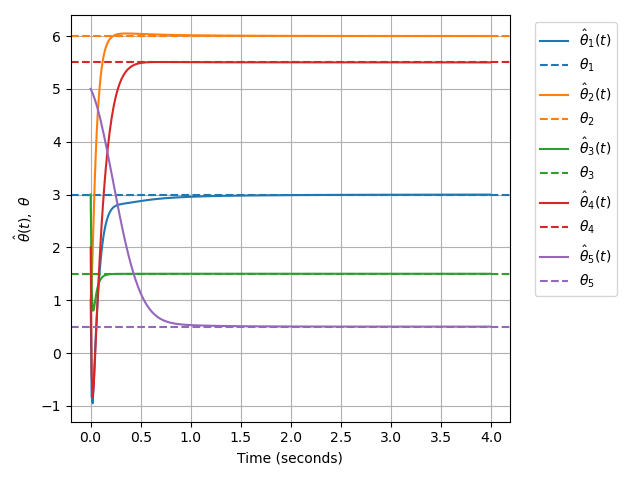}
	\end{minipage}
	\caption{Consensus error of states and parameter estimation.}
	\label{ep}
\end{figure}

\begin{figure}[h]
	\centering
	\begin{minipage}{0.45\textwidth}
		\centering
		\includegraphics[width=\textwidth]{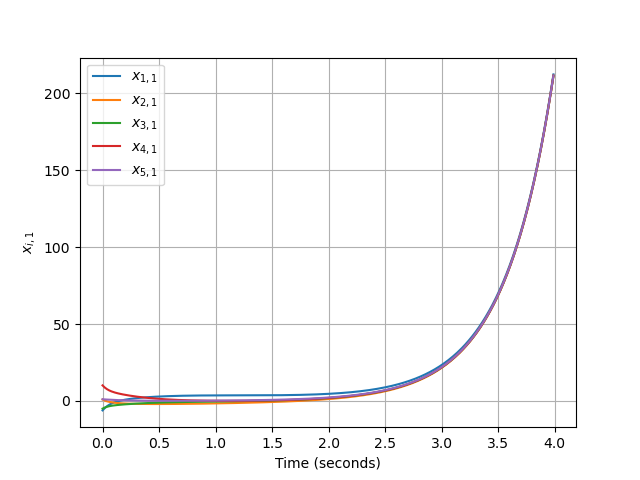}
	\end{minipage}
	\hfill
	\begin{minipage}{0.45\textwidth}
		\centering
		\includegraphics[width=\textwidth]{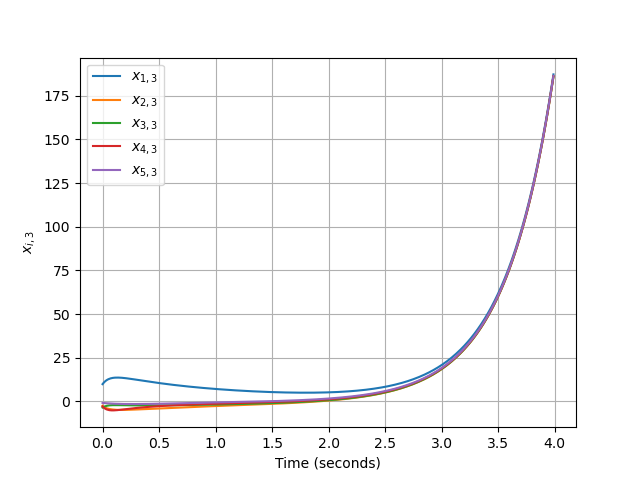}
	\end{minipage}
	\caption{Consensus of $x_{i,1}$ and $x_{i,3}$.}
	\label{c}
\end{figure}

Next we test the control \eqref{eq-2-12} involving the quantization operator. We set $\sigma=5$ for the quantization level in \eqref{eq-2-3}. The graph of consensus error and the convergence of parameters to references with quantization are presented in Figure \ref{epq} as we proved in Theorem \ref{thm-2}. It shows that the consensus and parameter estimation are acheived up to an error. The graphs of the first and third coordinates are depicted in Figure \ref{cq}. 

\begin{figure}[h]
	\centering
	\begin{minipage}{0.45\textwidth}
		\centering
		\includegraphics[width=\textwidth]{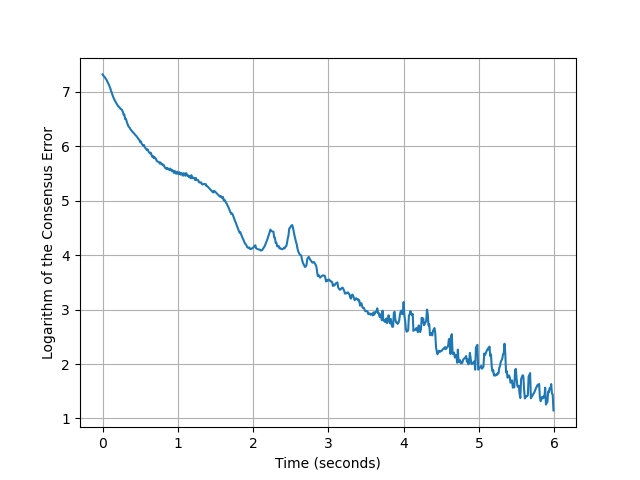}
	\end{minipage}
	\hfill
	\begin{minipage}{0.45\textwidth}
		\centering
		\includegraphics[width=\textwidth]{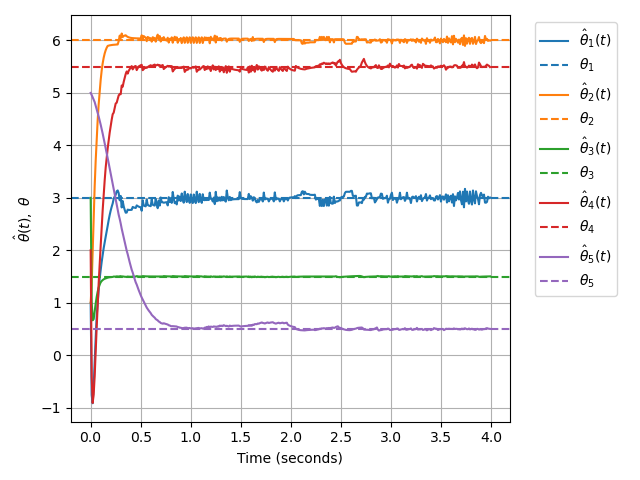}
	\end{minipage}
	\caption{Consensus error and parameter estimation with quantization $\sigma=5$.}
	\label{epq}
\end{figure}

\begin{figure}[h]
	\centering
	\begin{minipage}{0.45\textwidth}
		\centering
		\includegraphics[width=\textwidth]{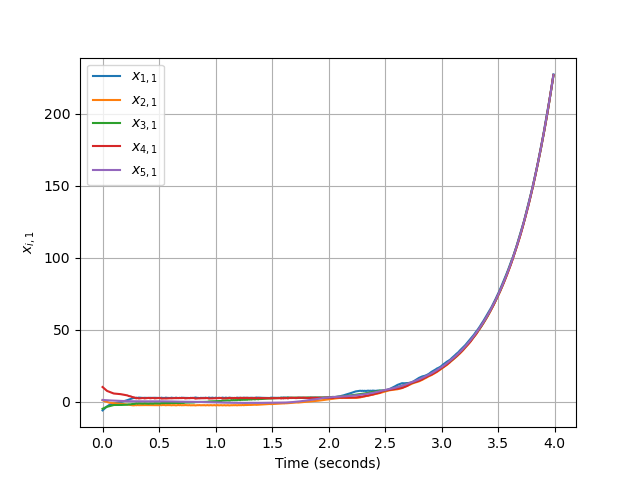}
	\end{minipage}
	\hfill
	\begin{minipage}{0.45\textwidth}
		\centering
		\includegraphics[width=\textwidth]{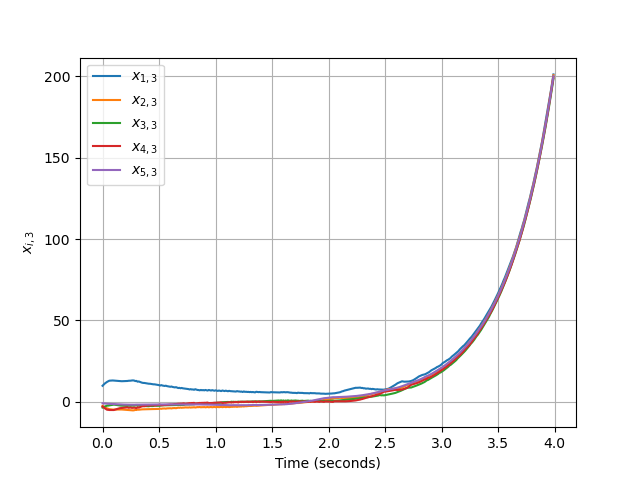}
	\end{minipage}
	\caption{Consensus of $x_{i,1}$ and $x_{i,3}$ with quantization $\sigma=5$.}
	\label{cq}
\end{figure}

Additionally, we test the dynamic \eqref{eq-2-1} with control \eqref{eq-2-12} for different levels of quantization $\sigma = \{10,15\}$. The consensus error and parameter estimation with quantization level $\sigma=10$ are presented in Figure \ref{ep10} and Figure \ref{ep15} depicts the case with $\sigma=15$. Figure \ref{ep10} and Figure \ref{ep15} show that the quantity $\sum_{i,j \in N} ||x_i(t) - x_j(t) ||^2$ converges exponentially fast up to an error whose size is proportional to the size of $\sigma >0$ and the parameter estimation are acheived up to a small error.

\begin{figure}[h]
	\centering
	\begin{minipage}{0.45\textwidth}
		\centering
		\includegraphics[width=\textwidth]{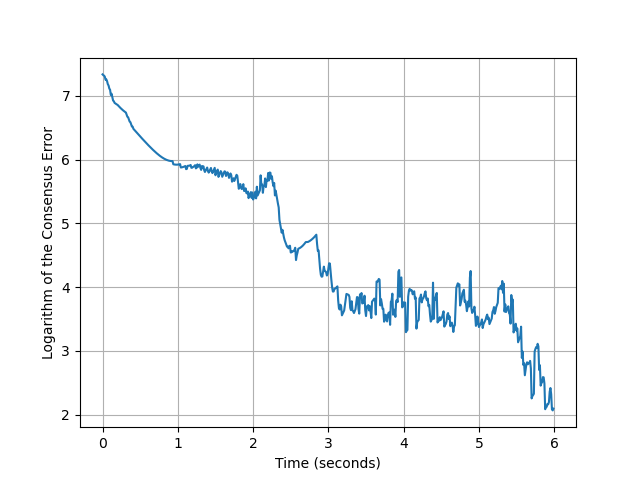}
	\end{minipage}
	\hfill
	\begin{minipage}{0.45\textwidth}
		\centering
		\includegraphics[width=\textwidth]{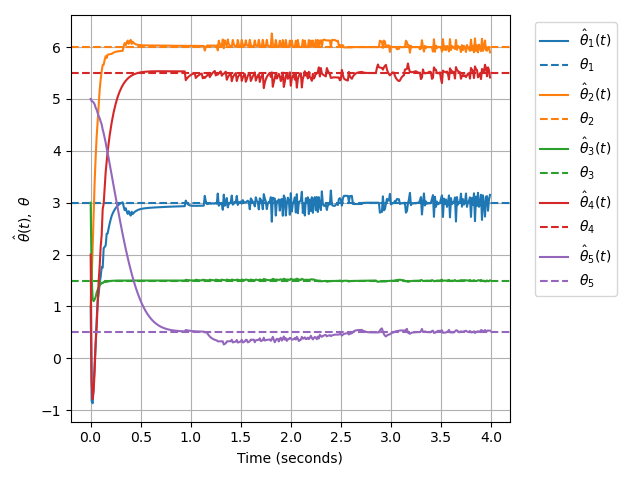}
	\end{minipage}
	\caption{Consensus error and parameter estimation with quantization $\sigma = 10$.}
	\label{ep10}
\end{figure}

\begin{figure}[h]
	\centering
	\begin{minipage}{0.45\textwidth}
		\centering
		\includegraphics[width=\textwidth]{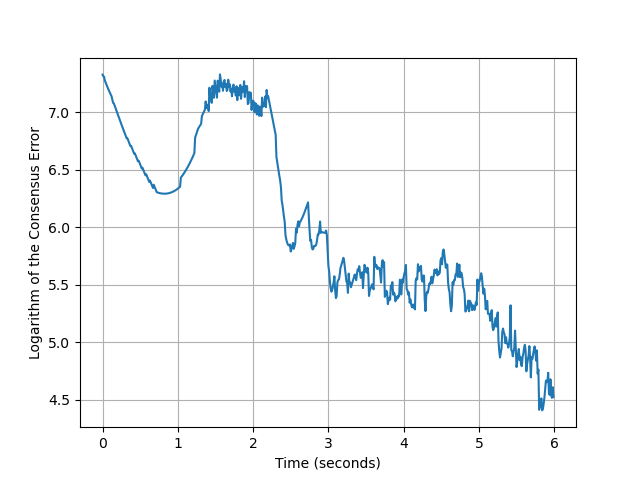}
	\end{minipage}
	\hfill
	\begin{minipage}{0.45\textwidth}
		\centering
		\includegraphics[width=\textwidth]{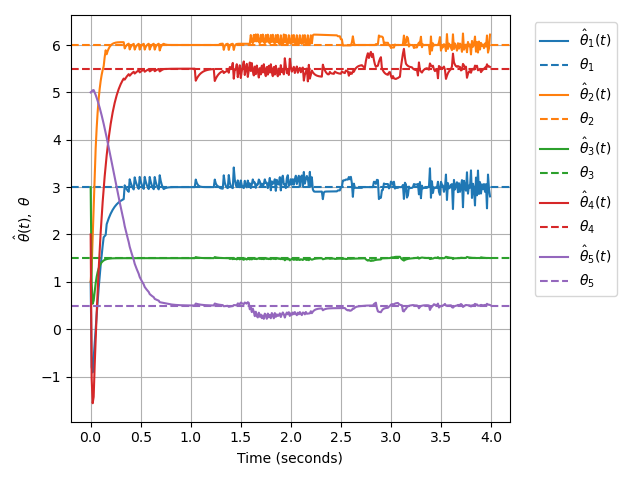}
	\end{minipage}
	\caption{Consensus error and parameter estimation with quantization $\sigma = 15$.}
	\label{ep15}
\end{figure}

\section{Conclusion}
In this study, we demonstrated that by introducing the concurrent learning technique, it is possible to achieve exponential convergence in adaptive consensus algorithms. By leveraging historical data, concurrent learning enhances parameter estimation accuracy and control performance. Additionally, we applied this technique in scenarios involving a uniform quantizer. The effectiveness of this approach is further validated by simulation results.





\end{document}